\documentclass[conference,onecolumn]{IEEEtran}

\usepackage[cp1250]{inputenc}	
\usepackage{amsmath}			
\usepackage{amsfonts}
\usepackage{amssymb}
\usepackage{amsthm}
\usepackage[T1]{fontenc}
\usepackage[stable]{footmisc}
\usepackage{animate}
\usepackage{psfrag}
\usepackage{fancybox}
\usepackage{bm} 
\usepackage{subfigure}
\usepackage{listings}
\usepackage[ruled,vlined,algosection]{algorithm2e}
\usepackage{float}

\usepackage{multirow}

\theoremstyle{plain}
\newtheorem{theorem}{Theorem}
\newtheorem{lemma}[theorem]{Lemma}
\newtheorem{corollary}[theorem]{Corollary}

\theoremstyle{definition}
\newtheorem{definition}{Definition}
\newtheorem{assumption}{Assumption}

\theoremstyle{remark}

\renewcommand{\vec}[1]{\ensuremath{\mathbf{#1}}}

\begin{document}

\title{Gradient of the Value Function in\\ Parametric Convex Optimization Problems}

\author{\IEEEauthorblockN{
Mato Baoti\'{c}}

\IEEEauthorblockA{Faculty of Electrical Engineering and Computing\\
University of Zagreb, Croatia \\
Email: mato.baotic@fer.hr}}

\maketitle

\begin{abstract}
We investigate the computation of the gradient of the value function in parametric convex optimization problems. We derive general expression for the gradient of the value function in terms of the cost function, constraints and Lagrange multipliers. In particular, we show that for the strictly convex parametric quadratic program the value function is continuously differentiable (denoted C$^{1}$) at every point in the interior of feasible space for which the Linear Independent Constraint Qualification holds.
\end{abstract}

\IEEEpeerreviewmaketitle

\section{Notation}
In general, we use bold letters (e.g., $\vec{a}$, $\vec{A}$) to denote vectors and matrices, and calligraphic letters (e.g., $\mathcal{A}$) to denote sets. With $\mathbb{N}$, $\mathbb{R}$, $\mathbb{R}^{n}$ and $\mathbb{R}^{m\times n}$ we denote the set of integers, real numbers, $n$-dimensional real (column) vectors, and $m\times n$ real matrices, respectively. Furthermore, $\vec{I}_{n}$ denotes the identity matrix in $\mathbb{R}^{n\times n}$, while $\vec{0}_{m}$ denotes the vector in $\mathbb{R}^{m}$ with all elements equal $0$, and $\vec{1}_{m}$ denotes the vector in $\mathbb{R}^{m}$ with all elements equal $1$. For a set $\mathcal{A}$ the set of all subsets of $\mathcal{A}$ is denoted with $2^{\mathcal{A}}$. For a set $\mathcal{A}\subset\mathbb{R}^{n}$ we denote with $\mathrm{cl}(\mathcal{A})$ the closure of $\mathcal{A}$ and with $\mathrm{int}(\mathcal{A})$ the interior of $\mathcal{A}$.
For matrix $\vec{A}\in\mathbb{R}^{m\times n}$, with $\vec{A}^{\top}$ we denote its transpose. For symmetric matrix $\vec{A}=\vec{A}^{\top} \in\mathbb{R}^{n\times n}$, with $\vec{A}\succ 0$ we state that it is positive definite, and with $\vec{A}\succeq 0$ that is is positive semidefinite. For a vector $\vec{x}\in\mathbb{R}^{n}$ with $x_i$ we denote $i$-th element of $\vec{x}$, $i\in\{1,\ldots, n\}$. For a matrix $\vec{A}\in\mathbb{R}^{m\times n}$
with $A_{i,j}$ we denote the element of $\vec{A}$ in $i$-th row and $j$-th column, with $\vec{A}_{i}$ we denote the $i$-th row, with $\vec{A}_{\bullet,j}$ we denote the $j$-th column, $i\in\{1,\ldots, m\}$, $j\in\{1,\ldots, n\}$, with $\vec{A}_{\mathcal{E}}$ we denote the matrix formed by rows of $\vec{A}$ indexed by $\mathcal{E}\subseteq\{1,\ldots,m\}$.

Let $\vec{h}: \mathbb{R}^{n} \rightarrow \mathbb{R}^{m}$ be a (vector valued) function,
\begin{equation*}
\vec{h}(\vec{x}) =
\left[\begin{array}{c}
{h}_{1}(\vec{x})\\
\vdots\\
{h}_{m}(\vec{x})
\end{array} \right]
=
\left[\begin{array}{c}
{h}_{1}(x_{1}, \ldots, x_{n})\\
\vdots\\
{h}_{m}(x_{1}, \ldots, x_{n})
\end{array} \right].
\end{equation*}
If $\vec{h}$ is differentiable at $\vec{x}_0\in\mathbb{R}^{n}$ then we can define $\dfrac{\partial \vec{h}}{\partial \vec{x}}(\vec{x}_0)$ as a \emph{matrix} in $\mathbb{R}^{m\times n}$
\begin{equation}\label{eq:Jacobian}
\dfrac{\partial \vec{h}}{\partial \vec{x}}(\vec{x}_0) :=
\left[\begin{array}{ccc}
\dfrac{\partial {h}_1}{\partial x_1}(\vec{x}_0) & \cdots &
\dfrac{\partial {h}_1}{\partial x_n}(\vec{x}_0)\\
\vdots & \ddots & \vdots \\
\dfrac{\partial {h}_m}{\partial x_1}(\vec{x}_0) & \cdots &
\dfrac{\partial {h}_m}{\partial x_n}(\vec{x}_0)
\end{array} \right],
\end{equation}
where
\begin{equation*}
\dfrac{\partial {h}_i}{\partial x_j}(\vec{x}_0):=
\left.\dfrac{\partial {h}_i(\vec{x})}{\partial x_j}\right|_{\vec{x}=\vec{x}_0}, \quad \forall i\in\{1,\ldots, m\}, \forall j\in\{1,\ldots, n\}.
\end{equation*}
The expression \eqref{eq:Jacobian} is also called \emph{Jacobian matrix} of $\vec{h}$ at $\vec{x}_0$.
Note that for a scalar function $V: \mathbb{R}^{n} \rightarrow \mathbb{R}$ that is differentiable at $\vec{x}_0$, the Jacobian matrix becomes a \emph{row vector}
\begin{equation*}
\dfrac{\partial {V}}{\partial \vec{x}}(\vec{x}_0) =
\left[\begin{array}{ccc}
\dfrac{\partial {V}}{\partial x_1}(\vec{x}_0) &
\cdots &
\dfrac{\partial {V}}{\partial x_n}(\vec{x}_0)
\end{array}\right].
\end{equation*}
The gradient of $V$ computed at $\vec{x}_0$, denoted as $\nabla_{\vec{x}}V(\vec{x}_0)$, is a \emph{column vector} in $\mathbb{R}^{n}$ defined as
\begin{equation*}
\nabla_{\vec{x}}V(\vec{x}_0)
:= \left[\dfrac{\partial {V}}{\partial \vec{x}}(\vec{x}_0)\right]^{\top}.
\end{equation*}
Note that the gradient $\nabla_{\vec{x}}V(\vec{x}_0)$ is well defined when $V$ is differentiable at $\vec{x}_0$ (i.e., $V$ does not have to be differentiable everywhere, it is sufficient for it to be differentiable locally).

\section{Parametric convex optimization problem}
%
Consider the following parametric convex optimization problem:
\begin{equation}\label{eq:mpProblem}
\begin{array}{rcl}
V(\vec{x}):=&\min\limits_{\vec{z}} & f(\vec{z},\vec{x})\\[1ex]
&\mathrm{s.t.}& \vec{g}(\vec{z},\vec{x})\leq \vec{0}_{m},
\end{array}
\end{equation}
where $f: \mathbb{R}^{s} \times \mathbb{R}^{n} \rightarrow \mathbb{R}$ and $\vec{g}: \mathbb{R}^{s} \times \mathbb{R}^{n} \rightarrow \mathbb{R}^{m}$ are \emph{convex} and \emph{continuously differentiable} functions on their domains, $\vec{x}\in\mathbb{R}^{n}$ is the parameter, $\vec{z}\in\mathbb{R}^{s}$ is the optimization vector, while $n\in\mathbb{N}$, $s\in \mathbb{N}$, and $m\in\mathbb{N}$, denote the number of parameters, optimization variables, and constraints, respectively.

The function $f(\vec{z},\vec{x})$ is called the \emph{cost function}, $\vec{g}(\vec{z}, \vec{x})$ are referred to as \emph{constraints}, and function $V(\vec{x}): \mathbb{R}^{n} \rightarrow \mathbb{R}$ is called the \emph{value function}. With $\mathcal{X}\subseteq \mathbb{R}^{n}$ we denote the \emph{set of feasible parameters},
\begin{equation}\label{eq:Xfeas}
\mathcal{X}:=\{\vec{x}\in \mathbb{R}^{n} \ | \
\exists \vec{z}\in\mathbb{R}^{s} : \vec{g}(\vec{z},\vec{x})\leq \vec{0}_{m} \}.
\end{equation}
Since $\vec{g}(\vec{z},\vec{x})$ is a convex function, it follows that $\mathcal{X}\subseteq\mathbb{R}^{n}$ is a closed convex set.\footnote{The set $\mathcal{X}$ is projection (convexity preserving operation) of an intersection (convexity preserving operation) of a finite number of sublevel sets of convex functions (convexity preserving operation) $g_i(\vec{z},\vec{x})$, $i=1,\ldots,m$.}

In the rest of the paper we use the following standing assumption.
\begin{assumption}\label{ass:WellDefined}
The set of feasible parameters, $\mathcal{X}$, is a full-dimensional set in $\mathbb{R}^{n}$, and optimization problem \eqref{eq:mpProblem} is well defined for all $\vec{x}\in\mathcal{X}$, i.e., $V(\vec{x})$ is finite and minimum is achieved with some $\vec{z}^{*}(\vec{x})$.
\end{assumption}

Now we can define $\mathcal{Z}: \mathcal{X} \rightrightarrows 2^{\mathbb{R}^s}$, a point-to-set mapping that maps each $\vec{x}\in\mathcal{X}$ to a \emph{set of feasible optimization variables}
\begin{equation}\label{eq:Zfeas}
\mathcal{Z}(\vec{x}):=\{\vec{z}\in \mathbb{R}^{s} \ | \
\vec{g}(\vec{z},\vec{x})\leq \vec{0}_{m} \},
\end{equation}
as well as $\mathcal{Z}^{*}: \mathcal{X} \rightrightarrows 2^{\mathbb{R}^s}$, the point-to-set mapping that maps each $\vec{x}\in\mathcal{X}$ to an \emph{optimal set}
\begin{equation}\label{eq:Zopt}
\mathcal{Z}^{*}(\vec{x}):=\{\bar{\vec{z}}\in \mathcal{Z}(\vec{x}) \ | \
f(\bar{\vec{z}},\vec{x}) \leq f(\vec{z},\vec{x})
\ \forall \vec{z} \in \mathcal{Z}(\vec{x}) \}.
\end{equation}
Under Assumption~\ref{ass:WellDefined}, by using convexity of $\vec{g}(\vec{z},\vec{x})$ and $f(\vec{z},\vec{x})$, one can easily show that $\mathcal{Z}(\vec{x})$ and $\mathcal{Z}^{*}(\vec{x})$ are closed convex sets for all $\vec{x}\in\mathcal{X}$.

In the following we use the notion of an \emph{optimizer} (function) for the problem \eqref{eq:mpProblem}, $\vec{z}^{*}: \mathcal{X} \rightarrow \mathbb{R}^{s}$,
\begin{equation}\label{eq:zopt}
\vec{z}^{*}(\vec{x})\in\mathcal{Z}^{*}(\vec{x}).
\end{equation}
In general, $\mathcal{Z}^{*}(\vec{x})$ is not a singleton, and expression \eqref{eq:zopt} states that $\vec{z}^{*}(\vec{x})$ is \emph{choosing} one particular value for each $\vec{x}$. We do not limit ourselves in the way this choice is being made, however, we will later put some conditions on the properties of $\vec{z}^{*}(\cdot)$.

In this paper we are interested in computation of a gradient of the value function, $\nabla_{\vec{x}}V(\vec{x})$, for all $\vec{x}$ in interior of $\mathcal{X}$.

\section{Gradient of the value function}
For problem \eqref{eq:mpProblem} we introduce a dual problem
\begin{equation}\label{eq:dual}
V_{\mathrm{D}}(\vec{x}) := \ \max\limits_{\boldsymbol{\lambda}\geq \vec{0}_{m}} \
\inf\limits_{\vec{z}} \ L(\vec{z},\boldsymbol{\lambda},\vec{x}),
\end{equation}
where $L:\mathbb{R}^{s} \times \mathbb{R}^{m} \times \mathbb{R}^{n} \rightarrow \mathbb{R}$ is the Lagrangian function
\begin{equation}\label{eq:Lagrangian}
L(\vec{z},\boldsymbol{\lambda},\vec{x}):=
f(\vec{z},\vec{x}) + \boldsymbol{\lambda}^{\top}\vec{g}(\vec{z},\vec{x}),
\end{equation}
and $\boldsymbol{\lambda}\in\mathbb{R}^{m}$ is the vector of dual variables, i.e., Lagrange multipliers. Similarly as in the case of \eqref{eq:mpProblem}, we denote with $\boldsymbol{\lambda^{*}}(\vec{x})$ an optimal value of dual variables, which is one choice out of all available optimizers of \eqref{eq:dual} at $\vec{x}\in\mathcal{X}$.

\begin{assumption}\label{ass:StrongDuality}
Strong duality holds for the problem \eqref{eq:mpProblem} and its dual \eqref{eq:dual} at $\bar{\vec{x}}\in\mathcal{X}$, i.e.,
\begin{equation}\label{eq:SD}
V(\bar{\vec{x}}) = V_{\mathrm{D}}(\bar{\vec{x}}).
\end{equation}
\end{assumption}

\noindent Note that there is variety of (sufficient) regularity conditions (e.g., Slater condition, Linear Independence Constraint Qualification) which guarantee that the convex problem \eqref{eq:mpProblem} satisfies Assumption~\ref{ass:StrongDuality}.

Under Assumption~\ref{ass:StrongDuality} we know that Karush-Kuhn-Tucker (KKT) first order necessary conditions must be satisfied for an optimal solution $(\vec{z}^{*},\boldsymbol{\lambda}^{*})$ of primal problem \eqref{eq:mpProblem} and dual problem \eqref{eq:dual} at $\bar{\vec{x}}\in\mathcal{X}$
\begin{equation}\label{eq:KKTOPT}
\nabla_{\vec{z}} f(\vec{z}^{*},\bar{\vec{x}})+\sum\limits_{i=1}^{m}\lambda_{i}^{*}\nabla_{\vec{z}} g_{i}(\vec{z}^{*},\bar{\vec{x}}) = \vec{0}_{s},
\end{equation}
\begin{equation}\label{eq:KKTPF}
\vec{g}(\vec{z}^{*},\bar{\vec{x}}) \leq \vec{0}_{m},
\end{equation}
\begin{equation}\label{eq:KKTDF}
\boldsymbol{\lambda}^{*} \geq \vec{0}_{m},
\end{equation}
\begin{equation}\label{eq:KKTCS}
\lambda_{i}^{*} g_{i}(\vec{z}^{*},\bar{\vec{x}}) = 0, \ i=1,\ldots,m.
\end{equation}

To emphasize that $\vec{z}^{*}$ and $\boldsymbol{\lambda}^{*}$ depend on parameter $\bar{\vec{x}}\in\mathcal{X}$, we write the stationarity condition \eqref{eq:KKTOPT} in the following form
\begin{equation}\label{eq:KKT1}
\dfrac{\partial f}{\partial \vec{z}^{*}} (\vec{z}^{*}(\bar{\vec{x}}),\bar{\vec{x}}) +
[\boldsymbol{\lambda}^{*}(\bar{\vec{x}})]^{\top}
\dfrac{\partial \vec{g}}{\partial \vec{z}^{*}}(\vec{z}^{*}(\bar{\vec{x}}),\bar{\vec{x}}) = \vec{0}_{s}^{\top}.
\end{equation}
Under Assumption~\ref{ass:StrongDuality}, at $(\vec{z}^{*}(\bar{\vec{x}}),\boldsymbol{\lambda}^{*}(\bar{\vec{x}}))$ the value function $V$ has the same value as the Lagrangian function $L$, i.e.,
\begin{equation}\label{eq:VL}
V(\bar{\vec{x}})=L(\vec{z}^{*}(\bar{\vec{x}}),\boldsymbol{\lambda}^{*}(\bar{\vec{x}}),\bar{\vec{x}})=
f(\vec{z}^{*}(\bar{\vec{x}}),\bar{\vec{x}}) + [\boldsymbol{\lambda}^{*}(\bar{\vec{x}})]^{\top}\vec{g}(\vec{z}^{*}(\bar{\vec{x}}),\bar{\vec{x}}).
\end{equation}

%
Let $\vec{z}^{*}(\cdot)$ and $\boldsymbol{\lambda}^{*}(\cdot)$ be continuously differentiable functions at $\bar{\vec{x}}\in\mathcal{X}$. By differentiating \eqref{eq:VL} we get
\begin{equation*}
\begin{array}{rcl}
\dfrac{\partial V}{\partial \vec{x}}(\bar{\vec{x}}) &=&
\dfrac{\partial f}{\partial \vec{z}^{*}}(\vec{z}^{*}(\bar{\vec{x}}),\bar{\vec{x}})
\cdot \dfrac{\partial \vec{z}^{*}}{\partial \vec{x}}(\bar{\vec{x}})
+\dfrac{\partial f}{\partial \vec{x}}(\vec{z}^{*}(\bar{\vec{x}}),\bar{\vec{x}})
+[\vec{g}(\vec{z}^{*}(\bar{\vec{x}}),\bar{\vec{x}})]^{\top}
\dfrac{\partial \boldsymbol{\lambda}^{*}}
{\partial \vec{x}}(\bar{\vec{x}}) + \\[2ex]
& & + [\boldsymbol{\lambda}^{*}(\bar{\vec{x}})]^{\top}
\left(\dfrac{\partial \vec{g}}{\partial \vec{z}^{*}}(\vec{z}^{*}(\bar{\vec{x}}),\bar{\vec{x}})
\cdot \dfrac{\partial \vec{z}^{*}}{\partial \vec{x}}(\bar{\vec{x}})
+\dfrac{\partial \vec{g}}{\partial \vec{x}}(\vec{z}^{*}(\bar{\vec{x}}), \bar{\vec{x}})\right).
\end{array}
\end{equation*}
By rearranging the terms we have
\begin{equation*}
\begin{array}{rcl}
\dfrac{\partial V}{\partial \vec{x}}(\bar{\vec{x}}) &=&
\dfrac{\partial f}{\partial \vec{x}}(\vec{z}^{*}(\bar{\vec{x}}),\bar{\vec{x}})
+[\vec{g}(\vec{z}^{*}(\bar{\vec{x}}),\bar{\vec{x}})]^{\top}
\dfrac{\partial \boldsymbol{\lambda}^{*}}
{\partial \vec{x}}(\bar{\vec{x}})
+[\boldsymbol{\lambda}^{*}(\bar{\vec{x}})]^{\top}
\dfrac{\partial \vec{g}}{\partial \vec{x}}(\vec{z}^{*}(\bar{\vec{x}}), \bar{\vec{x}}) + \\[2ex]
& & +
\left(
\dfrac{\partial f}{\partial \vec{z}^{*}}(\vec{z}^{*}(\bar{\vec{x}}),\bar{\vec{x}})
+[\boldsymbol{\lambda}^{*}(\bar{\vec{x}})]^{\top}
\dfrac{\partial \vec{g}}{\partial \vec{z}^{*}}(\vec{z}^{*}(\bar{\vec{x}}),\bar{\vec{x}})
\right)\dfrac{\partial \vec{z}^{*}}{\partial \vec{x}}(\bar{\vec{x}}),
\end{array}
\end{equation*}
which combined with the KKT condition \eqref{eq:KKT1} finally gives
\begin{equation}\label{eq:partV}
\dfrac{\partial V}{\partial \vec{x}}(\bar{\vec{x}}) =
\dfrac{\partial f}{\partial \vec{x}}(\vec{z}^{*}(\bar{\vec{x}}),\bar{\vec{x}})
+[\vec{g}(\vec{z}^{*}(\bar{\vec{x}}),\bar{\vec{x}})]^{\top}
\dfrac{\partial \boldsymbol{\lambda}^{*}}
{\partial \vec{x}}(\bar{\vec{x}})
+[\boldsymbol{\lambda}^{*}(\bar{\vec{x}})]^{\top}
\dfrac{\partial \vec{g}}{\partial \vec{x}}(\vec{z}^{*}(\bar{\vec{x}}), \bar{\vec{x}}).
\end{equation}
Clearly, the expression \eqref{eq:partV} is the transpose of the gradient of the value function. We repeat again that this result for $\nabla_{\vec{x}}V(\bar{\vec{x}})$ was derived under assumption of continuous differentiability of $\vec{z}^{*}(\cdot)$ and $\boldsymbol{\lambda}^{*}(\cdot)$ at $\bar{\vec{x}}\in\mathcal{X}$.

\begin{definition}[Active set]
Consider a feasible point $(\vec{z},\vec{x})$ of problem \eqref{eq:mpProblem}. We say that the $i$-th constraint $g_i(\vec{z},\vec{x})\leq 0$, $i\in\{1,\ldots,m\}$, is
\emph{active} at $(\vec{z},\vec{x})$ if $g_i(\vec{z},\vec{x})= 0$. If $g_i(\vec{z},\vec{x})< 0$ we say that the $i$-th constraint is \emph{inactive} at $(\vec{z},\vec{x})$. The active set for $\bar{\vec{x}}\in\mathcal{X}$, denoted $\mathcal{A}(\bar{\vec{x}})$, is the set of all constraints that are active at $(\vec{z}^{*}(\bar{\vec{x}}),\bar{\vec{x}})$
\begin{equation}\label{eq:AS}
\mathcal{A}(\bar{\vec{x}}):=\{i\in\{1,\ldots,m\} \ | \
g_i(\vec{z}^{*}(\bar{\vec{x}}),\bar{\vec{x}})=0 \}.
\end{equation}
\end{definition}
\noindent Note that: i) $\emptyset$ may be an active set for some $\bar{\vec{x}}\in\mathcal{X}$, ii) there may be some subsets of $\{1,\ldots,m\}$ which are not active sets for any $\bar{\vec{x}}\in\mathcal{X}$, iii) the number of (different) active sets of problem \eqref{eq:mpProblem} is at most $2^{m}$.

\begin{definition}[Critical region]
Let $\mathcal{E}$ be an active set of problem \eqref{eq:mpProblem} (i.e., there exists $\bar{\vec{x}}\in\mathcal{X}$ such that $\mathcal{A}(\bar{\vec{x}}) = \mathcal{E}$). Then the following set
\begin{equation}\label{eq:CR}
\mathcal{R}(\mathcal{E}):=\{\vec{x}\in\mathcal{X} \ | \ \mathcal{A}(\vec{x})= {\mathcal{E}} \}
\end{equation}
is called the \emph{critical region} of problem \eqref{eq:mpProblem} (associated with an active set $\mathcal{E}$).
\end{definition}

\begin{definition}[Neighboring regions]
Two critical regions, $\mathcal{R}_1$ and $\mathcal{R}_2$, are called \emph{neighboring} if intersection of their closures is a non-empty set
\begin{equation*}
\mathrm{cl}(\mathcal{R}_1) \cap \mathrm{cl}(\mathcal{R}_1)\neq \emptyset.
\end{equation*}
\end{definition}

In the following we will use $\mathbb{A}$ to denote the set of all active sets of problem \eqref{eq:mpProblem} that generate full-dimensional critical regions
\begin{equation}\label{eq:ndimCR}
\mathbb{A}:=\{\mathcal{E} \subseteq \{1,\ldots,m\} \ | \
\exists \bar{\vec{x}}\in\mathcal{X}: \mathcal{A}(\bar{\vec{x}})=\mathcal{E}, \
\mathcal{R}(\mathcal{E}) \ \textrm{is full-dimensional set} \},
\end{equation}
and note the following statement is true
\begin{equation}
\mathcal{X} = \bigcup_{\mathcal{E} \in\mathbb{A}} \mathrm{cl}(\mathcal{R}(\mathcal{E})),
\end{equation}
because $\mathcal{X}$ is a closed full-dimensional set.

\begin{lemma}\label{lem:gradV}
Let Assumption~\ref{ass:WellDefined} and Assumption~\ref{ass:StrongDuality} hold. Let $\mathcal{E}\in\mathbb{A}$
and let the solution to the problem \eqref{eq:mpProblem} and its dual \eqref{eq:dual}, $\vec{z}^{*}(\vec{x})$ and $\boldsymbol{\lambda}^{*}(\vec{x})$, be continuously differentiable functions on $\mathrm{int}(\mathcal{R}(\mathcal{E}))$. Then for all $\vec{x} \in \mathrm{int}(\mathcal{R}(\mathcal{E}))$ the gradient of the value function can be computed as follows
\begin{equation}\label{eq:gradV}
\nabla_{\vec{x}}V(\vec{x}) =
\left[\dfrac{\partial f}{\partial \vec{x}}(\vec{z}^{*}(\vec{x}),\vec{x})\right]^{\top}
+\left[\dfrac{\partial \vec{g}}{\partial \vec{x}}(\vec{z}^{*}(\vec{x}), \vec{x})\right]^{\top}
\boldsymbol{\lambda}^{*}(\vec{x}),
\end{equation}
i.e.,
\begin{equation}\label{eq:gradVAS}
\nabla_{\vec{x}}V(\vec{x}) =
\left[\dfrac{\partial f}{\partial \vec{x}}(\vec{z}^{*}(\vec{x}),\vec{x})\right]^{\top}
+\sum\limits_{i\in\mathcal{E}}
\left[ \dfrac{\partial g_{i}}{\partial \vec{x}}(\vec{z}^{*}(\vec{x}), \vec{x})\right]^{\top}
\lambda_{i}^{*}(\vec{x}).
\end{equation}
\end{lemma}
\begin{proof} From complementarity slackness conditions \eqref{eq:KKTCS} it follows that for all $j\in \{1,\ldots,m\}\setminus \mathcal{E}$ one has $\lambda_{j}^{*}(\vec{x})=0$ for all $\vec{x}\in\mathrm{int}(\mathcal{R}_{\mathcal{E}})$. Since for all $i\in \mathcal{E}$ one has $g_{i}(\vec{z}^{*}(\vec{x}),\vec{x})=0$ for all $\vec{x}\in\mathrm{int}(\mathcal{R}_{\mathcal{E}})$, clearly the second term in \eqref{eq:partV} is equal to $0$, while the third term keeps only components indexed by $\mathcal{E}$. Therefore \eqref{eq:partV} reduces to \eqref{eq:gradV} and \eqref{eq:gradVAS}.
\end{proof}

\begin{theorem}\label{thm:gradVcont}
Let Assumption~\ref{ass:WellDefined} and Assumption~\ref{ass:StrongDuality} hold. Let the solution to the problem \eqref{eq:mpProblem} and its dual \eqref{eq:dual}, $\vec{z}^{*}(\vec{x})$ and $\boldsymbol{\lambda}^{*}(\vec{x})$, be: i) continuous functions on $\mathrm{int}(\mathcal{X})$, and ii) continuously differentiable functions on $\mathrm{int}(\mathcal{R}(\mathcal{E}))$ for all $\mathcal{E}\in\mathbb{A}$. Then the gradient of the value function, $\nabla_{\vec{x}}V(\vec{x})$, is given by expression \eqref{eq:gradV} for all $\vec{x} \in \mathrm{int}(\mathcal{X})$. Furthermore, $\nabla_{\vec{x}}V(\vec{x})$ is a continuous function for all $\vec{x} \in \mathrm{int}(\mathcal{X})$.
\end{theorem}

\begin{proof}
Since $\vec{z}^{*}(\vec{x})$ and $\boldsymbol{\lambda}^{*}(\vec{x})$ are continuously differentiable for all $\vec{x}\in\mathrm{int}(\mathcal{R}(\mathcal{E}))$, and $f(\vec{z}^{*}(\vec{x}),\vec{x})$ and $\vec{g}(\vec{z}^{*}(\vec{x}),\vec{x})$ are continuously differentiable on $\mathrm{int}(\mathcal{X})$, we can use Lemma~\ref{lem:gradV}. From \eqref{eq:gradVAS} we see that $\nabla_{\vec{x}}V(\vec{x})$ is a continuous function on $\mathrm{int}(\mathcal{R}(\mathcal{E}))$ because its computation involves composition, multiplication and addition of a finite number of continuous functions.




What is left to prove is that for any point $\vec{x}_{\mathrm{B}}$ on the boundary of the closure of two full-dimensional neighboring critical regions we get the same value for the gradient from both sides. Let $\mathcal{R}_{1}$ and $\mathcal{R}_{2}$ be two full-dimensional neighboring critical regions, with active sets $\mathcal{E}_{1}$ and $\mathcal{E}_{2}$, and let $\vec{x}_{\mathrm{B}}\in \mathrm{int}(\mathcal{X})$ be a point on their shared boundary,
\begin{equation*}
\vec{x}_{\mathrm{B}}\in \mathrm{cl}(\mathcal{R}_1) \cap \mathrm{cl}(\mathcal{R}_1).
\end{equation*}

Since $\nabla_{\vec{x}}V(\vec{x})$ is continuous on $\mathrm{int}(\mathcal{R}_k)$, $k=1,2$, we can define its limit when approaching $\vec{x}_{\mathrm{b}}$ from interior of $\mathrm{int}(\mathcal{R}_k)$
\begin{equation*}
\nabla_{\vec{x}}V_{k}(\vec{x}_{\mathrm{B}}):=
\lim\limits_{\begin{array}{c}\vec{x}\in\mathrm{int}(\mathcal{R}_{k})\\
\vec{x}\rightarrow \vec{x}_{\mathrm{B}} \end{array}} \nabla_{\vec{x}}V(\vec{x})
=
\lim\limits_{\begin{array}{c}\vec{x}\in\mathrm{int}(\mathcal{R}_{k})\\
\vec{x}\rightarrow \vec{x}_{\mathrm{B}} \end{array}}
\left[\dfrac{\partial f}{\partial \vec{x}}(\vec{z}^{*}(\vec{x}),\vec{x})\right]^{\top}
+\sum\limits_{i\in\mathcal{E}_k}
\left[ \dfrac{\partial g_{i}}{\partial \vec{x}}(\vec{z}^{*}(\vec{x}), \vec{x})\right]^{\top}
\lambda_{i}^{*}(\vec{x}).
\end{equation*}
By using the fact that $\vec{z}^{*}(\vec{x})$ and $\boldsymbol{\lambda}^{*}(\vec{x})$ are continuous on $\mathrm{int}(\mathcal{X})$, and that $f(\vec{z}^{*}(\vec{x}),\vec{x})$, and $\vec{g}(\vec{z}^{*}(\vec{x}),\vec{x})$ are continuously differentiable functions on $\mathrm{int}(\mathcal{X})$, the above expression becomes
\begin{equation*}
\nabla_{\vec{x}}V_{k}(\vec{x}_{\mathrm{B}})=
\left[\dfrac{\partial f}{\partial \vec{x}}(\vec{z}^{*}(\vec{x}_{\mathrm{B}}),\vec{x}_{\mathrm{B}})\right]^{\top}
+\sum\limits_{i\in\mathcal{E}_k}
\left[ \dfrac{\partial g_{i}}{\partial \vec{x}}(\vec{z}^{*}(\vec{x}_{\mathrm{B}}), \vec{x}_{\mathrm{B}})\right]^{\top}
\lambda_{i}^{*}(\vec{x}_{\mathrm{B}}), \qquad k=1,2.
\end{equation*}
Hence we get
\begin{equation}
\nabla_{\vec{x}}V_{1}(\vec{x}_{\mathrm{B}})-\nabla_{\vec{x}}V_{2}(\vec{x}_{\mathrm{B}})=
\sum\limits_{i\in(\mathcal{E}_1\setminus \mathcal{E}_2)}
\left[ \dfrac{\partial g_{i}}{\partial \vec{x}}(\vec{z}^{*}(\vec{x}_{\mathrm{B}}), \vec{x}_{\mathrm{B}})\right]^{\top}
\lambda_{i}^{*}(\vec{x}_{\mathrm{B}})
-
\sum\limits_{j\in(\mathcal{E}_2\setminus \mathcal{E}_1)}
\left[ \dfrac{\partial g_{j}}{\partial \vec{x}}(\vec{z}^{*}(\vec{x}_{\mathrm{B}}), \vec{x}_{\mathrm{B}})\right]^{\top}
\lambda_{j}^{*}(\vec{x}_{\mathrm{B}}).
\end{equation}
Note that for all $i\in \{1,\ldots,m\}\setminus \mathcal{E}_2$ and all $\bar{\vec{x}}\in\mathcal{R}_2$, by complementarity slackness \eqref{eq:KKTCS}, we have $\lambda_{i}^{*}(\bar{\vec{x}})=0$. Since $\lambda_{i}^{*}(\vec{x})$ is continuous function on $\mathrm{int}(\mathcal{X})$ this implies that $\lambda_{i}^{*}(\vec{x}_{\mathrm{B}})=0$ for all $i\in \{1,\ldots,m\}\setminus \mathcal{E}_2$. With similar reasoning we deduce that $\lambda_{j}^{*}(\vec{x}_{\mathrm{B}})=0$ for all $j\in \{1,\ldots,m\}\setminus \mathcal{E}_1$. Therefore, all terms on the right hand side of the above expression vanish, and we have proven the claim.

%
%
\end{proof}

Note that we exclude points on the boundary of $\mathcal{X}$ from Theorem~\ref{thm:gradVcont} because the gradient of a function is well defined only for points in strict interior of function's domain.

\section{Multi-parametric Quadratic Program}
Consider the following special case of problem \eqref{eq:mpProblem}, the so-called multi-parametric quadratic programming (mpQP) problem:
\begin{equation}\label{eq:mpQP}
\begin{array}{rcl}
V(\vec{x}):=&\min\limits_{\vec{z}} & \dfrac{1}{2}\vec{z}^{\top}\vec{H}\vec{z}\\[1ex]
&\mathrm{s.t.}& \vec{G}\vec{z}\leq \vec{W} + \vec{S}\vec{x},
\end{array}
\end{equation}
where $\vec{H}=\vec{H}^{\top}\succ 0$, $\vec{H}\in\mathbb{R}^{s\times s}$, $\vec{G}\in\mathbb{R}^{m\times s}$, $\vec{W}\in\mathbb{R}^{m}$ and $\vec{S}\in\mathbb{R}^{m\times n}$, and $\mathcal{X}=\{\vec{x}\in\mathbb{R}^{n} \ | \ \exists \vec{z}\in\mathbb{R}^{s} : \vec{G}\vec{z}\leq \vec{W}+\vec{S}\vec{x}\}$ is a full-dimensional set.

It can be shown that the dual for problem \eqref{eq:mpQP} is
\begin{equation}\label{eq:mpQPdual}
\begin{array}{rcl}
V_{\mathrm{D}}(\vec{x}):=&\max\limits_{\boldsymbol{\lambda}} &
-\dfrac{1}{2}\boldsymbol{\lambda}^{\top}\vec{G}\vec{H}^{-1}\vec{G}^{\top}\boldsymbol{\lambda}
- \boldsymbol{\lambda}^{\top}(\vec{W}+\vec{S}\vec{x})\\[1ex]
&\mathrm{s.t.}& \boldsymbol{\lambda}\geq \vec{0}_{m}.
\end{array}
\end{equation}

\begin{corollary}\label{cor:mpQP}
Let Linear Independent Constraint Qualification (LICQ) hold for all $\vec{x}\in\mathcal{X}$ in mpQP \eqref{eq:mpQP}. Then the value function $V(\vec{x})$ is continuously differentiable on $\mathrm{int}(\mathcal{X})$, and its gradient is given by the following expression
\begin{equation}\label{eq:mpQPVgrad}
\nabla_{\vec{x}}V(\vec{x})=-\vec{S}^{\top}\boldsymbol{\lambda}^{*}(\vec{x}), \quad \forall \vec{x} \in \mathrm{int}(\mathcal{X}).
\end{equation}
\end{corollary}
\begin{proof}
Since $\vec{H}$ is positive definite and constraints in \eqref{eq:mpQP} are linear, then Assumption~\ref{ass:WellDefined} and Assumption~\ref{ass:StrongDuality} are satisfied. Furthermore, it is known\footnote{See details in \cite{Bemporad:2002} and \cite{Tondel:2003}.} that, under LICQ, the following holds for \eqref{eq:mpQP}--\eqref{eq:mpQPdual}: the optimal solutions ($\vec{z}^{*}(\vec{x})$, $\boldsymbol{\lambda}^{*}(\vec{x})$) are continuous functions on $\mathrm{int}(\mathcal{X})$, with (one) affine expression for each critical region; critical regions are polyhedra (some boundaries of critical region may be open and some closed) and they partition the feasible space $\mathcal{X}$; the value function $V(\vec{x})$ is continuous and convex function, with (one) quadratic expression for each critical region.

Since all conditions of Theorem~\ref{thm:gradVcont} are satisfied we can utilize \eqref{eq:gradV} to compute $\nabla_{\vec{x}}V(\vec{x})$ for all $\vec{x}\in \mathrm{int}(\mathcal{X})$. By noting that $f(\vec{z},\vec{x})=\frac{1}{2}\vec{z}^{\top}\vec{H}\vec{z}$ and $\vec{g}(\vec{z},\vec{x}) = \vec{G}\vec{z}-\vec{S}\vec{x}-\vec{W}$, we get

\begin{equation*}
\nabla_{\vec{x}}V(\vec{x}) =
\left[\dfrac{\partial f}{\partial \vec{x}}(\vec{z}^{*}(\vec{x}),\vec{x})\right]^{\top}
+\left[\dfrac{\partial \vec{g}}{\partial \vec{x}}(\vec{z}^{*}(\vec{x}), \vec{x})\right]^{\top}
\boldsymbol{\lambda}^{*}(\vec{x})=-\vec{S}^{\top}\boldsymbol{\lambda}^{*}(\vec{x}),
\quad \forall \vec{x} \in \mathrm{int}(\mathcal{X}),
\end{equation*}
which is continuous function on $\mathrm{int}(\mathcal{X})$, since $\boldsymbol{\lambda}^{*}(\vec{x})$ is continuous on $\mathrm{int}(\mathcal{X})$.
\end{proof}

The result of Corollary~\ref{cor:mpQP} can be made even more precise if we are interested in the expression of $\nabla_{\vec{x}}V(\vec{x})$ for one critical region $\mathcal{R}(\mathcal{E})$, where $\mathcal{E}\in\mathbb{A}$,
\begin{equation}
\nabla_{\vec{x}}V(\vec{x})=-\vec{S}_{\mathcal{E}}^{\top}
\boldsymbol{\lambda}_{\mathcal{E}}^{*} (\vec{x}),
\quad \forall \vec{x} \in \mathcal{R}(\mathcal{E}) \cap \mathrm{int}(\mathcal{X}),
\end{equation}
which combined with
\begin{equation}\label{eq:mpQPlambdastar}
\boldsymbol{\lambda}_{\mathcal{E}}^{*}=
-(\vec{G}_{\mathcal{E}}\vec{H}^{-1}\vec{G}_{\mathcal{E}}^{\top})^{-1}
(\vec{W}_{\mathcal{E}}+\vec{S}_{\mathcal{E}}\vec{x}),
\quad \forall \vec{x} \in \mathcal{R}(\mathcal{E}),
\end{equation}
finally gives
\begin{equation}\label{eq:mpQPVgradCR}
\nabla_{\vec{x}}V(\vec{x})=-\vec{S}_{\mathcal{E}}^{\top}
(\vec{G}_{\mathcal{E}}\vec{H}^{-1}\vec{G}_{\mathcal{E}}^{\top})^{-1}
(\vec{W}_{\mathcal{E}}+\vec{S}_{\mathcal{E}}\vec{x}),
\quad \forall \vec{x} \in \mathcal{R}(\mathcal{E}) \cap \mathrm{int}(\mathcal{X}).
\end{equation}
Note that in the case when $\emptyset \in \mathbb{A}$ (i.e., there is a full-dimensional critical region with unconstrained solution, $\mathcal{R}(\emptyset)$) the expression \eqref{eq:mpQPVgradCR} must be modified to
\begin{equation}\label{eq:mpQPVgradUnc}
\nabla_{\vec{x}}V(\vec{x})=\vec{0}_{n},
\quad \forall \vec{x} \in \mathcal{R}(\emptyset)\cap \mathrm{int}(\mathcal{X}).
\end{equation}

We point out that Corollary~\ref{cor:mpQP} can be proven more directly, without referral to the generalized result of Theorem~\ref{thm:gradVcont}.
\begin{proof}[Second proof of Corollary~\ref{cor:mpQP}]
By strong duality of \eqref{eq:mpQP} and \eqref{eq:mpQPdual} we get
\begin{equation}\label{eq:alt1}
V(\vec{x}) = -\dfrac{1}{2}[\boldsymbol{\lambda}^{*}(\vec{x})]^{\top} \vec{G}\vec{H}^{-1}\vec{G}^{\top}\boldsymbol{\lambda}^{*}(\vec{x})
- [\boldsymbol{\lambda}^{*}(\vec{x})]^{\top}(\vec{W}+\vec{S}\vec{x}).
\end{equation}
By differentiation of \eqref{eq:alt1} (assuming $\vec{x}\in\mathrm{int}(\mathcal{R}(\mathcal{E}))$ for some $\mathcal{E}\in\mathbb{A}$) we get

\begin{equation}\label{eq:alt2}
\begin{array}{rcl}
\nabla_{\vec{x}}V(\vec{x}) &=&
- \left[\dfrac{\partial \boldsymbol{\lambda}^{*}}{\partial \vec{x}}(\vec{x})\right]^{\top} \vec{G}\vec{H}^{-1}\vec{G}^{\top}\boldsymbol{\lambda}^{*}(\vec{x})
- \left[\dfrac{\partial \boldsymbol{\lambda}^{*}}{\partial \vec{x}}(\vec{x})\right]^{\top} (\vec{W}+\vec{S}\vec{x})
- \vec{S}^{\top}\boldsymbol{\lambda}^{*}(\vec{x}) =\\[3ex]
&=&
\left[\dfrac{\partial \boldsymbol{\lambda}^{*}}{\partial \vec{x}}(\vec{x})\right]^{\top}
(- \vec{G}\vec{H}^{-1}\vec{G}^{\top}\boldsymbol{\lambda}^{*}(\vec{x}) - \vec{W}-\vec{S}\vec{x})
- \vec{S}^{\top}\boldsymbol{\lambda}^{*}(\vec{x}).
\end{array}
\end{equation}
By KKT condition \eqref{eq:KKTOPT} we have
\begin{equation}\label{eq:alt3}
\vec{z}^{*}(\vec{x}) = -\vec{H}^{-1}\vec{G}^{\top}\boldsymbol{\lambda}^{*}(\vec{x}),
\end{equation}
which combined with \eqref{eq:alt2} gives
\begin{equation}\label{eq:alt4}
\begin{array}{rcl}
\nabla_{\vec{x}}V(\vec{x}) &=&
- \vec{S}^{\top}\boldsymbol{\lambda}^{*}(\vec{x}) +
\left[\dfrac{\partial \boldsymbol{\lambda}^{*}}{\partial \vec{x}}(\vec{x})\right]^{\top}
(\vec{G}\vec{z}^{*} - \vec{W}-\vec{S}\vec{x})=\\
&=&
- \vec{S}^{\top}\boldsymbol{\lambda}^{*}(\vec{x})
+\sum\limits_{i\in\mathcal{E}} \left[\dfrac{\partial {\lambda}_{i}^{*}}{\partial \vec{x}}(\vec{x})\right]^{\top}
\underbrace{(\vec{G}_{i}\vec{z}^{*} - \vec{W}_{i}-\vec{S}_{i}\vec{x})}_{0}
+\sum\limits_{i\in\{1,\ldots, m\}\setminus \mathcal{E}}
\underbrace{\left[\dfrac{\partial {\lambda}_{i}^{*}}{\partial \vec{x}}(\vec{x})\right]^{\top}}_{\vec{0}_{n}}
(\vec{G}_{i}\vec{z}^{*} - \vec{W}_{i}-\vec{S}_{i}\vec{x})=\\
&=&
- \vec{S}^{\top}\boldsymbol{\lambda}^{*}(\vec{x}).
\end{array}
\end{equation}
Expression \eqref{eq:alt4} was derived for $\vec{x}\in\mathrm{int}(\mathcal{R}(\mathcal{E}))$, but, since it is the same for all full-dimensional critical regions, and since we have $\lambda^{*}(\vec{x})$ continuous on $\mathrm{int}(\mathcal{X})$, we have completed the proof.
\end{proof}

\bibliographystyle{plain}

\end{document}